\documentclass[twoside, 11pt]{amsart}

\usepackage{mathtools}
\usepackage{amscd}
\usepackage{amsmath}
\usepackage{amssymb}
\usepackage{amsthm}
\usepackage[english]{babel}
\usepackage{cancel}
\usepackage{tikz}
\usepackage{tikz-cd}
\usetikzlibrary{matrix, arrows}
\usepackage{url}
\usepackage[all]{xy}

\usepackage{amsfonts}

\usepackage{mathrsfs}

\usepackage[colorlinks,final,hyperindex]{hyperref}

\usepackage[left=2.5cm,right=2.5cm,top=2.5cm,bottom=2.5cm]{geometry}

\hypersetup{citecolor=BleuMinuit, linkcolor=bordeau, filecolor=black, urlcolor=BleuMinuit}

\allowdisplaybreaks

\newtheorem{lemma}{Lemma}[section]
\newtheorem{theorem}[lemma]{Theorem}

\newtheorem{corollary}[lemma]{Corollary}
\newtheorem{proposition}[lemma]{Proposition}
\newtheorem{conjecture}[lemma]{Conjecture}

\theoremstyle{definition}
\newtheorem{definition}[lemma]{Definition}
\newtheorem{remark}[lemma]{\sc Remark}
\newtheorem{example}[lemma]{\sc Example}
\newtheorem{notation}[lemma]{\sc Notation}

\newcommand{\nocontentsline}[3]{}
\newcommand{\tocless}[2]{\bgroup\let\addcontentsline=\nocontentsline#1{#2}\egroup}

\author[Emprin]{Coline Emprin}
\author[Hunter]{Dana Hunter}
\author[Livernet]{Muriel Livernet}
\author[Vespa]{Christine Vespa}
\author[Zakharevich]{Inna Zakharevich}

\title{A prop structure on partitions}
\date{\today}

\address{Coline Emprin, D\'epartement de math\'ematiques et applications, \'Ecole normale sup\'erieure, 45 rue d’Ulm, 75230 Paris Cedex 05, France}

\email{\href{mailto:coline.emprin@ens.psl.eu}{coline.emprin@ens.psl.eu}}

\address{Dana Hunter, Kalamazoo College, 1200 Academy Street, Kalamazoo, Michigan, 	49006-3295, USA}

\email{\href{mailto:Dana.Hunter@kzoo.edu}{Dana.Hunter@kzoo.edu}}

\address{Muriel Livernet, Univ. Paris Cit\'e, Institut de Math\'ematiques de Jussieu-Paris
	Rive Gauche, CNRS, SU, DMA, ENS-PSL,	Paris, France}

\email{\href{mailto:livernet@math.univ-paris-diderot.fr}{livernet@math.univ-paris-diderot.fr}}

\address{Christine Vespa, Aix Marseille Univ, CNRS, I2M, Marseille, France}

\email{\href{mailto:christine.vespa@univ-amu.fr}{christine.vespa@univ-amu.fr}}

\address{Inna Zakharevich, Math. Dept., Cornell University, Ithaca, NY, USA}

\email{\href{mailto:zakh@math.cornell.edu}{zakh@math.cornell.edu}}



\begin{document}

\newcommand{\End}{\operatorname{End}}
\newcommand{\PROP}{\mathbf{PROP}}
\newcommand{\K}{\mathbb{K}} 
\newcommand{\Ss}{\mathbb{S}}  

\newcommand{\C}{\mathcal{C}}
\newcommand{\D}{\mathcal{D}}
\newcommand{\E}{\mathcal{E}}
\newcommand{\F}{\mathcal{F}}
\newcommand{\Operad}{\mathbf{Operad}}
\newcommand{\Co}{\mathcal{C}}
\renewcommand{\P}{\mathcal{P}}     
\newcommand{\SP}{s \mathcal{P}} 
\newcommand{\I}{\mathcal{I}}
\newcommand{\N}{\mathbb{N}}
\newcommand{\V}{\mathcal{V}} 
\newcommand{\Com}{\mathcal{C}om}
\newcommand{\SCom}{s\mathcal{C}om} 
\newcommand{\Ext}{\operatorname{Ext}}
\newcommand{\sgn}{\operatorname{sgn}}  
\newcommand{\Surj}{\mathrm{Surj}}
\newcommand{\Sh}{\mathrm{Sh}}
\newcommand{\Part}{\mathrm{Part}}
\newcommand{\Comp}{\mathrm{Comp}}
\newcommand{\Kar}{\mathrm{Kar}} 
\newcommand{\OSurj}{\mathrm{Surj}_{\mathsf{or}}} 
\newcommand{\Fgr}{{\mathcal{F}(\mathbf{gr})}}
\newcommand{\abK}{{\mathfrak{a}}_{\K}}

\newcommand{\tb}{\ast}

\newcommand{\proj}{\mathrm{proj}} 

\definecolor{bordeau}{rgb}{0.5,0,0}
\definecolor{BleuTresFonce}{rgb}{0.215, 0.215, 0.36}
\definecolor{BleuMinuit}{RGB}{0, 51, 102}
\definecolor{rose}{RGB}{235, 62, 124}

\newcommand{\Coline}[1]{\textcolor{rose}{#1}}
\newcommand{\Muriel}[1]{\textcolor{red}{#1}}
\newcommand{\Christine}[1]{\textcolor{cyan}{#1}}
\newcommand{\Dana}[1]{\textcolor{violet}{#1}}

\begin{abstract}
Motivated by its link with functor homology, we study the prop freely generated by the operadic suspension of the operad $Com$. We exhibit a particular family of generators, for which the composition and the symmetric group actions admit simple descriptions. We highlight associated subcategories of its Karoubi envelope which allows us to compute extensions groups between simple functors from free groups. We construct a particular prop structure on partitions whose composition corresponds to the Yoneda product of extensions between exterior power functors.
\end{abstract}

	\maketitle

\makeatletter
\def\@tocline#1#2#3#4#5#6#7{\relax
	\ifnum #1>\c@tocdepth 
	\else
	\par \addpenalty\@secpenalty\addvspace{#2}%
	\begingroup \hyphenpenalty\@M
	\@ifempty{#4}{%
		\@tempdima\csname r@tocindent\number#1\endcsname\relax
	}{%
		\@tempdima#4\relax
	}%
	\parindent\z@ \leftskip#3\relax \advance\leftskip\@tempdima\relax
	\rightskip\@pnumwidth plus4em \parfillskip-\@pnumwidth
	#5\leavevmode\hskip-\@tempdima
	\ifcase #1
	\or\or \hskip 1em \or \hskip 2em \else \hskip 3em \fi%
	#6\nobreak\relax
	\hfill\hbox to\@pnumwidth{\@tocpagenum{#7}}\par
	\nobreak
	\endgroup
	\fi}

\newcommand{\enableopenany}{%
	\@openrightfalse%
}
\makeatother
	
	\setcounter{tocdepth}{1}
	\tableofcontents

\section*{\textcolor{bordeau}{Introduction}}

The notions of props and operads arose in the work of Mac Lane \cite{ML65}, in the aim of encoding algebraic structures. While operads encode products with a single output, props allow for working with algebraic structures involving operations with multiple outputs. Such structures include Hopf algebras, Frobenius algebras, or Lie bialgebras, which arouse interest after the discovery of quantum groups for instance, see \cite{Dr1, Dr2}. A prop is a symmetric monoidal category with objects the natural numbers and whose symmetric monoidal structure $\otimes$ is given by the sum of integers on objects. The notions of operads and props are intrinsically related: one can consider the prop freely generated by a given operad and to any prop, one can associate its underlying operad. Working with set operads, it is well known that the prop freely generated by the terminal operad $\Com$ is the category of surjections $\Surj$, see e.g. \cite{HPV}.  \\

The present paper focuses on the graded linear prop $\mathcal{E}$ freely generated by the operadic suspension of the commutative operad. In this context, the suspension gives rise to signs. Let $\K$ denote the underlying ground field of characteristic zero. As a graded $\K$-vector space, the space $\E^{\bullet}(m,n)$ is concentrated in degree $m-n$, where we have an isomorphism \[\E^{m-n}(m,n) \simeq \K[\Surj(m,n)] \ .\] In \cite{KV23}, the authors exhibit a system of generators having the advantage that the monoidal structure simply corresponds to the disjoint union of sets. However, the composition involves signs. In Section \ref{1}, we provide another system of generators which has the advantage that the left and right actions of the symmetric groups are by sign, and the composition agrees with the set composition of surjections.  
Considering idempotents of the symmetric groups, one can construct various categories out of the prop $\mathcal{E}$, defined as subcategories of its Karoubi envelope. The general construction is recalled in Section \ref{2}. In particular, for any prop $\C$, one can consider a certain subcategory $\Lambda\C$ of the Karoubi envelope of $\C$. It is equivalent to a category $\C_\Lambda$ where the vector space of morphisms is obtained from $\C$ by taking the quotient by the symmetric group actions. Its composition is described in Theorem \ref{T:catquotient}. In the case of the prop $\mathcal{E}$, the space $\E^{\bullet}_{\Lambda}(m,n)$ is concentrated in degree $m-n$ and is spanned by partitions of $m$ into $n$ parts. The category structure, coming from that of $\mathcal{E}$, is described in Section \ref{prop-partition}. Nonetheless, the monoidal product of $\mathcal{E}$, does not induce any prop structure on $\E_{\Lambda}$. In Theorem \ref{T:main}, we introduce a different monoidal product which turns the category $\E_\Lambda$ into a prop and thus leads to a particular prop structure on partitions. \medskip

Our interest in these props comes from their link with extensions between functors from free groups which is explored in Section \ref{Ext}. In fact, functor homology turned out to be a useful tool for computing stable homology with twisted coefficients of various families of groups, and in particular of automorphism groups of free groups $\mathrm{Aut}(\mathbb{Z}^{*n})$, for $n\in \mathbb{N}$. Djament proved in \cite{Dja19} that stable cohomology of $\mathrm{Aut}(\mathbb{Z}^{*n})$ with coefficients given by a reduced polynomial covariant functor is governed by Ext-groups in the category $\Fgr$ of functors from finitely generated free groups to $\K$-vector spaces. From this perspective, \cite{vespa15} gives an explicit computation of the graded $\mathbb{K}$-vector spaces \[\Ext^\bullet_{\Fgr}((T^n\circ\mathfrak{a}) \otimes \K,(T^m\circ \mathfrak{a})\otimes \K)  \ ,\] where $\mathfrak{a}$ is the abelianization functor and $T^n$ is the $n$-th tensor power functor. Together with the Yoneda product and external product of extension, this family forms a prop which is shown to be isomorphic to the prop $\mathcal{E}$. This prop structure was leveraged to give explicit computations of stable homology of $\mathrm{Aut}(\mathbb{Z}^{*n})$ with coefficients given by particular contravariant functors, see \cite[Theorem~4]{vespa15}. The aformentionned results were extended in \cite{KV23, Dja19} in order to deal with bivariant coefficients. One can study extensions between other functors, by replacing $T^n$ by exterior power functors $\Lambda^n$ for example. By \cite{vespa15}, the extensions between these functors are concentrated in one degree and are spanned by partitions. A prop structure on these groups can thus be derived from that of $\mathcal{E}_{\Lambda}$. We conclude this paper by exploiting the construction of Section \ref{2} in order to give some other explicit computations of extension groups between simple functors.

 \bigskip 

\noindent \textbf{Notation.}

\begin{enumerate}
	\item Let $\K$ be a characteristic zero field and let $ \mathrm{gr Vect}_{\K}$ be the symmetric monoidal category of graded vector spaces over  $\K$.  We use the cohomological grading convention $V^\bullet$, and  the degree of an element $x$ in $V$  is denoted $\mathrm{d}(x)$. For every set $S$, we denote by $\K[S]$ the $\K$-vector space spanned by $S$. 
 	\item We denote by $|S|$ the cardinal of a finite set $S$. 
 
	\item The set of surjections from $\{1,\ldots,m\}$ to $\{1,\ldots,n\}$ is denoted $\Surj(m,n)$. 
  Composition of surjections  is denoted $\circ$.  Given $f\in\Surj(m,n)$ and $g\in\Surj(m',n')$ we denote by $f\times g$ the element in $
	\Surj(m+m',n+n')$ defined by
	\[ (f\times g)(i)=\begin{cases} f(i),& \text{ for } 1\leqslant  i\leqslant  m,\\
		g(i-m)+n,& \text{ for } m+1\leqslant  i\leqslant  m+m'.\end{cases}\]
	\item The symmetric group on $n$ letters is denoted $\mathbb{S}_n$. 
	\item $\tau_{i,i+1}$ denotes the transposition of $\mathbb{S}_n$ that permutes $i$ and $i+1$. $\epsilon (\sigma)$ denotes the sign of the permutation $\sigma$.
	\item A $(p,q)$-unshuffle is the inverse of a $(p,q)$-shuffle permutation, that is, a 
	permutation $\sigma$ such that $\sigma^{-1}(1)<\ldots< \sigma^{-1}(p)$ and $\sigma^{-1}(p+1)<\ldots<\sigma^{-1}(p+q)$.  Similarly we define a $(p_1,\ldots,p_n)$-unshuffle. For a surjection $f\in\Surj(m,n)$, we denote by $\Sh_f$ the set of $(p_1,\ldots,p_n)$-unshuffles with $p_i=|f^{-1}(i)|$.
	
		\item\label{N:dec} We denote by $\OSurj(m,n)$ the set of order-preserving surjections. For $f\in\Surj(m,n)$, there is a unique decomposition $f=s\circ \alpha$ with $s\in\OSurj(m,n)$ and $\alpha\in\Sh_f$.
		
    \item\label{N:part} A partition $\lambda$ of $m$ into $n$ parts is a sequence of positive integers $\lambda_1\geqslant \ldots\geqslant \lambda_n$ such that $\sum_i \lambda_i=m$. We denote by $\Part(m,n)$ the set of partitions of $m$ into $n$ parts. To a surjection $f\in\Surj(m,n)$, we can associate a partition of $m$ into $n$ parts given by ordering the cardinals of its fibers in the decreasing order. We denote by \[\proj:\Surj(m,n)\to\Part(m,n) \] this map. By linear extension we have a surjective morphism \[\proj:\K[\Surj(m,n)]\to\K[\Part(m,n)] \ .\]
    
\end{enumerate}

\bigskip 

\noindent \textbf{Acknowledgements.}
We would like to thank the Hausdorff Research Institute for Mathematics for hosting the Women in Topology IV workshop and for financial support. We would also like to thank the Foundation Compositio Mathematica, the Foundation Nagoya Mathematical Journal and the K-theory Foundation for financial support for this event.

\section{\textcolor{bordeau}{A graded linear prop spanned by surjections}}\label{1}

The aim of this section is to give an explicit description of the prop freely generated by the suspension of the commutative operad. We emphasize a choice of generators for which the symmetric group actions are given by the sign and which behave well with respect to the composition of maps.

\subsection{Recollections on props}

This section recalls the definition of a prop and the freely generated prop associated to an operad. We refer the reader to \cite{Mar08} for more details on props and to \cite{LV12} for more details on operads.

\begin{definition}[Graded linear prop]\label{D:prop}
A \emph{graded linear prop}, or simply prop, is a symmetric monoidal category $(\Co, \otimes, 1)$, enriched over the category of graded vector spaces $\mathrm{gr Vect}_{\K}$, with objects the natural numbers and whose symmetric monoidal structure $\otimes$ is given by the sum of integers on objects. In other words, a graded linear prop is the data of a collection $\{\Co(m,n) \}_{m,n \in \mathbb{N}}$ of graded $\K$-vector spaces
together with compatible morphisms:
\begin{itemize}
	\item[$\centerdot$] a vertical composition given by the categorical composition, \[\diamond:\; \Co(n,l) \otimes \Co(m,n) \longrightarrow \Co(m,l)\] 
	\item[$\centerdot$] an horizontal composition coming from the monoidal product, \[ \otimes:\; \Co(m,n)  \otimes \Co(m',n') \longrightarrow \Co(m + m', n +n') \] 
	\item[$\centerdot$] isomorphisms $s_{m,m'} \in \Co(m+m',m+m')$ such that we have \[(-1)^{\mathrm{d}(f) \mathrm{d}(g)} \left(g\otimes f \right) \diamond s_{m,m'} = s_{n,n'} \diamond \left(f\otimes g\right) , \] for all $f\in \Co(m,n)$ and all $g\in \Co(m',n')$. 
	\end{itemize}
Throughout the paper, we will use the terminology prop to refer to a graded linear prop.  
\end{definition}

\begin{remark}\label{phi}
	The isomorphisms $s_{m,m'}$ induce a morphism of $\mathbb K$-algebras $\varphi: \K[\mathbb{S}_n]\to\P(n,n)$. This gives a ${\K}[\mathbb{S}_m] $-right action and ${\K}[\mathbb{S}_n]$-left action on  $\P(m,n)$ defined for $\sigma\in\mathbb{S}_n, \tau \in \mathbb{S}_m,$ and $f \in\P(m,n)$ by\begin{equation}\label{E:action_comp}
	\sigma\cdot f \cdot\tau=\varphi(\sigma)\diamond f \diamond \varphi(\tau).
	\end{equation}
	In particular, one can find an equivalent definition in the literature of prop using only the action of the symmetric groups and its compatibility with vertical and horizontal compositions, see e.g. \cite{HR15}.
\end{remark}

\begin{remark}
For every prop $(\Co, \otimes, 1)$, the collection $\{\Co(n,1)\}_{n\in\mathbb N}$ forms an operad in $\mathrm{gr Vect}_{\K}$ where the composition maps are given for all $k,n_1, \dots, n_k \geqslant 0 $ by \[\Co(k,1)\otimes \Co(n_1,1) \otimes \ldots \otimes \Co(n_k,1) \to \Co(k,1)\otimes \Co(n_1+\ldots+n_k , k) \to \Co(n_1+\ldots+n_k,1) \ , \] where the first map is induced by the monoidal product $\otimes$ and the second map by the composition in the category $\Co$. We will refer to it as the \emph{underlying operad} of the prop $(\Co, \otimes, 1)$. This leads to a forgetful functor from the category of props to the one of operads. Its restriction to reduced props and operads admits a left ajoint $\Omega$ which associates to any reduced operad $\P$, the prop $\Omega \P$ freely generated by $\P$. 
\end{remark}

\begin{definition}[Freely generated prop]\label{omega P}
Let $\P$ be a reduced operad in $\mathrm{gr Vect}_{\K}$ (i.e. $\P(0)=0$). The graded linear prop $\Omega \P$ freely generated by $\P$ is given by the following data. 

\noindent $\centerdot$\emph{ The collection $\{\Omega \P(m,n)\}_{m,n \in \mathbb{N}}$ :}  It is defined by
		\begin{align*}
		\Omega\P(m,n)=&\bigoplus_{p_1+\ldots+p_n=m} \P(p_1)\otimes\ldots\otimes \P(p_n)
		\otimes_{\mathbb{S}_{p_1}\times\ldots\times\mathbb{S}_{p_n}}
		\K[\mathbb{S}_n] \\
		=&\bigoplus_{f\in\Surj(m,n)} \P(|f^{-1}(1)|)\otimes\ldots\otimes \P(|f^{-1}(n)|)\ . 
		\end{align*}

\noindent We denote by $\Omega \P_f$ the summand corresponding to a surjection $f \in \Surj(m,n)$ in the previous decomposition.

\noindent $\centerdot$\ \emph{The monoidal product} is obtained by concatenation, i.e. we consider \[x\otimes y\in \Omega\P_{f\times g} \ , \] for all $x \in \Omega\P_f$ and all $y \in \Omega\P_{g}$. \\ $\centerdot$\ \emph{The $\mathbb{S}_n$-left action} is defined for all $\sigma\in \mathbb{S}_n$, all $f\in \Surj(m,n)$ and all $x=x_1\otimes\ldots\otimes x_n\in\Omega\P_f$, where $x_i \in \P(|f^{-1}(i)|) $ by the formula
	\begin{equation}
\sigma\cdot x=\pm x_{\sigma^{-1}(1)}\otimes\ldots \otimes x_{\sigma^{-1}(n)} \in \Omega\P_{\sigma\circ f} \ ,
	\end{equation}
   where $\pm$ is the Koszul sign rule induced by the degrees of each element $x_i$. \\ $\centerdot$\ \emph{The $\mathbb{S}_m$-right action} is defined for $\tau \in \mathbb{S}_m$, $f\in \Surj(m,n)$ and $x=x_1\otimes\ldots\otimes x_n\in\Omega\P_f$ as follows. Decompose $f$  as $s\circ\alpha$ with $s\in\OSurj(m,n)$ and $\alpha\in\Sh_f$, and decompose
	$\alpha\circ\tau$ as $\alpha\circ\tau=(\sigma_1\times\ldots\times\sigma_n)\circ u$ with $\sigma_i\in\Ss_{|f^{-1}(i)|}$ and $u\in\Sh_f$.
	 Then, we have
  \begin{equation}
	(x_1\otimes\ldots\otimes x_n)\cdot \tau= \left( x_1\cdot \sigma_1 \right)\otimes\ldots\otimes \left(x_n\cdot \sigma_n \right) \in \Omega\P_{f\circ\tau}. 
 \end{equation}
	
		\noindent  $\centerdot$\ \emph{The composition morphisms} \[ \Omega\P(m,n)\otimes\Omega\P(l,m) \overset{\diamond}{\longrightarrow} \Omega\P(l,n) \] are described as follows. 	Let $s \in \OSurj(m,n)$ be an order-preserving surjection and let $g\in\Surj(l,m)$. Let us consider $x=x_1\otimes\ldots\otimes x_n \in \Omega \P_s$ and $y=y_1\otimes\ldots\otimes y_n \in\Omega\P_g$, with $y_i=y_i^1\otimes\ldots\otimes y_i^{p_i} $ where $p_i=|s^{-1}(i)|$. Then, we have 
	\begin{equation}\label{E:prop_comp}
	     x\diamond y= (-1)^{\omega} \gamma(x_1;y_1)\otimes\ldots \otimes\gamma(x_n;y_n) \in \Omega\P_{s\circ g}, 
	\end{equation}
  	where $\gamma$ denotes the composition in the operad $\P$ and $\omega$ is given by the Koszul sign rule, namely \[ \omega \coloneqq \sum_{i=1}^n \mathrm{d}(y_{i-1})\left(\mathrm{d}(x_{i})+\cdots + \mathrm{d}( x_n) \right).\] \\
    For $f\in\Surj(m,n)$, let us consider the decomposition $f=s\circ \alpha$ with $s\in \OSurj(m,n)$ and $\alpha\in\Sh_f$. Let us denote by $x_s \in \Omega\P_s $ the unique element satisfying $x_s \cdot \alpha=x \in \Omega\P_f$. Then, we have \[ x\diamond y=(x_s\cdot \alpha)\diamond y=x_s\diamond (\alpha\cdot y) \ ,\] which can be computed through the composition with an order-preserving surjection and the left action 
	of the symmetric group already defined.

\noindent $\centerdot$ The isomorphisms $s_{n,m}\in\Omega\P(n+m,n+m)$ are defined as $s_{n,m}={1}^{\otimes n+m}\in \Omega\P_\sigma$ 
where $\sigma$ is the unshuffle defined as $\sigma(i)=m+i$ for $1\leq i\leq n$ and $\sigma(i)=i-n$ for $n+1\leq i\leq n+m$.

\noindent $\centerdot$ The map $\varphi:\K[\Ss_n]\to \Omega\P(n,n)$ sends $\sigma$ to $1^{\otimes n}\in\Omega\P_\sigma$ where $1\in\P(1)$ denotes the unit of the operad $\P$.
\end{definition}

\subsection{The prop freely generated by the operadic suspension of $\Com$} \label{E}
Let us consider  $\SCom$ the operadic suspension of the commutative operad, see \cite[Section~7.2.2]{LV12} for more details. The graded operad $\SCom$ is generated by a single operation $\mu_2$, of degree $1$ and arity $2$, subject to the relation \[\mu_2\circ_1 \mu_2=-\mu_2\circ_2 \mu_2 \ ,\] and the action of $\mathbb{S}_2$ on $\mu_2$ is by sign. By using the conventions of \cite[Section~9]{KV23}, we have that $\SCom(n)=\mathbb K \mu_n\otimes\sgn_n$ is concentrated in degree $n-1$, and \[\mu_3 = \mu_2\circ_1 \mu_2= -\mu_2\circ_2 \mu_2 \ , \] 
so that

\[\mu_n\circ_i \mu_k=(-1)^{(i-1)(k-1)} \mu_{n+k-1} \ , \] where $\mu_1 = 1$ is the unit of the operad $sCom$. The operadic composition in $\SCom$ is then given by 
	
 \begin{equation*}
 \mu_n(\mu_{p_1},\ldots,\mu_{p_n})=(-1)^{\kappa(p_1,\ldots,p_n)} \mu_{p_1+\ldots+p_n} 
 \end{equation*}
 where 
\begin{equation}\label{E:kappa}
		\kappa(p_1,\ldots,p_n)=\sum_{j=1}^n (p_j-1)(p_1+\ldots +p_{j-1})=\sum_{k=1}^n p_k(p_{k+1}-1+\ldots+p_n-1) \ .
\end{equation}

\begin{notation}\label{N:E}
    Let us denote by $\E$ the graded linear prop $\Omega\SCom$ freely generated by the operad $\SCom$.
\end{notation}
\noindent By Definition \ref{omega P}, for $f\in\Surj(m,n)$, we have 
$$\E_f^\bullet=  
\left\{
\begin{array}{ll}
	\SCom(|f^{-1}(1)|)\otimes\ldots\otimes \SCom(|f^{-1}(n)|)  \simeq  \K & \mbox{if } \bullet = m-n  \text{ and } m\geqslant  n,\\
	0 &  \mbox{otherwise }  
	\end{array}
\right. 
$$
 as a graded vector space. In order to give an explicit description of this prop using Definition \ref{omega P} we need to choose, for any $f\in\Surj(m,n)$, a generator in the vector space $\E_f^{m-n}$. 
  In \cite{KV23}, the authors consider the generators $\mu_f:=\mu_{p_1}\otimes\ldots\otimes \mu_{p_n}$ of $\E_f^{m-n}$ for $f\in\Surj(m,n)$ and $p_i=|f^{-1}(i)|$. For these generators the composition $\diamond$  and the action of the symmetric groups give rise to complicated signs whereas the monoidal structure is simply given by concatenation without signs. More precisely, the left action by $\tau_{i,i+1}$ is given by
 \begin{equation}\label{E:prop_leftaction}
 \tau_{i,i+1}\cdot \mu_f=(-1)^{(p_i-1)(p_{i+1}-1)}\mu_{\tau\circ f}\ .
 \end{equation}
 For the right action, following Definition \ref{omega P}, we decompose $f$ as $f=s\circ\alpha$, with $s\in\OSurj(m,n)$ and $\alpha\in\Sh_f$. Given $\tau\in\Ss_m$, there exist $\sigma=(\sigma_1\times\ldots\times\sigma_n), \sigma_i\in\Ss_{p_i}$ and $u\in \Sh_f$ such that
 $\alpha\circ\tau=\sigma\circ u$. The right action on $\Omega\SCom$ is given by sign, so that
 \begin{equation}\label{E:prop_rightaction}
     \mu_f\cdot\tau=\epsilon(\sigma)\mu_{f\circ\tau} \ .
 \end{equation}

 We give below an alternative choice of generators in $\E_f^{m-n}$ having the advantage that the description of the composition $\diamond$ is easy and that the action of the symmetric groups on these generators on both sides is given by sign.

\begin{notation} \label{nu}
Let $f\in\Surj(m,n)$ be a surjection such that $p_i=|f^{-1}(i)|$ and let $\kappa(f)$ be the integer 
$\kappa(p_1,\ldots,p_n)$ defined in (\ref{E:kappa}). Let us denote \[ \nu_f \coloneqq \epsilon(\alpha) (-1)^{\kappa(f)} \mu_{p_1} \otimes \cdots \otimes \mu_{p_n}=\epsilon(\alpha) (-1)^{\kappa(f)}\mu_f \ , \] where $s \circ \alpha$ is the decomposition of $f$ with $s\in\OSurj(m,n)$ and $\alpha\in\Sh_f$. Let us note that $\nu_\sigma =\epsilon(\sigma) \mu_\sigma$ for all $\sigma \in \mathbb{S}_m$, or equivalently that $\varphi(\sigma)=\epsilon(\sigma)\nu_\sigma$. The degree is given by $\mathrm{d}(\nu_f)=m-n$ and \[\E_f^{m-n} \simeq \mathbb K \nu_f \ .\] 
\end{notation}


\noindent The following theorem describes the prop structure on $\E^\bullet$ in terms of the generators $\nu_f$. 

\begin{theorem}\label{P:easyprop}
Let $\E$ be the graded linear prop $\Omega\SCom$.

\begin{enumerate}
    \item[(a)]  The $\mathbb{S}_n$-left action and the $\mathbb{S}_m$-right action are given by \[\sigma\cdot \nu_f\cdot \tau =\epsilon(\sigma)\epsilon(\tau) \nu_{\sigma\circ f\circ\tau} \ ,\] for all $f\in \Surj(m,n)$, all $\sigma\in\mathbb{S}_n$ and all $\tau\in\mathbb{S}_m$.

\item[(b)] Let $s \in \OSurj(m,n)$ and $t \in\OSurj(m',n')$ be order-preserving surjections. The monoidal product is given by the concatenation up to a sign, i.e.  
\[\nu_s \otimes \nu_{t}=(-1)^{\mathrm{d}(\nu_t)m} \nu_{s\times t} \ .\] 
If $f$ and $g$ are not order-preserving surjections, the formula for $\nu_f \otimes \nu_g$ can be derived by decomposing 
$f$ and $g$ into order-preserving surjections and unshuffles.
    \item[(c)] The composition is defined for $f\in\Surj(m,n)$ and $h\in\Surj(l,m)$ by \[\nu_f\diamond \nu_h=\nu_{f\circ h} \ . \] 
\end{enumerate}

\end{theorem}

\begin{proof} 
Let $f \in \Surj(m,n)$ be a surjection which decomposes uniquely as $s \circ\alpha$ with $s\in\OSurj(m,n)$ and $\alpha\in\Sh_f$. Let us denote $p_i= |f^{-1}(i)| =|s^{-1}(i)|$. We have $\kappa(f)=\kappa(s)$. 

\noindent Let us prove Point (a). The $\mathbb{S}_m$-right action is given for $\tau \in \mathbb{S}_m$ as follows. Let us write $\alpha\circ\tau=\sigma\circ u$ with $\sigma=(\sigma_1\times\ldots\times\sigma_n), \sigma_i\in\Ss_{p_i}$ and $u\in \Sh_f$. By definition of $\mu_f$ and Relation (\ref{E:prop_rightaction}), we have
		\[\nu_f\cdot\tau=
  (-1)^{\kappa(f)}\epsilon(\alpha)\epsilon(\sigma)\mu_{f\circ\tau}\ .\] Since $s\circ\sigma=s$, we have $f\circ\tau= s\circ u$ and \[\mu_{f\circ\tau}=\epsilon(u)(-1)^{\kappa(f)}\nu_{f\circ\tau} \ .\] This leads to $\nu_f\cdot\tau=\epsilon(\tau) \nu_{f\circ\tau} \ . $

\noindent The $\mathbb{S}_n$-left action is given for $\sigma \in \mathbb{S}_n$ as follows. Let $s \in\OSurj(m,n)$ be an order preserving surjection such that $p_i=|s^{-1}(i)|$ and let us write $\sigma\circ s=t\circ \beta$ with $t\in\OSurj(m,n)$ and $\beta\in\Sh_{\sigma\circ s}.$ First, we claim that \[
    \epsilon(\sigma)\sigma\cdot\mu_s=\epsilon(\beta)(-1)^{\sum_{i=1}^np_i(\sigma(i)-i)}\mu_{\sigma\circ s} \]
Since this formula is clearly multiplicative, it is enough to prove it for $\sigma=\tau_{i,i+1}$. By Formula (\ref{E:prop_leftaction}), we have
\[\tau_{i,i+1}\cdot\mu_s=(-1)^{(p_i-1)(p_{i+1}-1)}\mu_{\tau_{i,i+1}\circ s} \ , \] and $\epsilon(\beta)=(-1)^{p_ip_{i+1}}$, since $\beta$ exchanges the fibers $i$ and $i+1$, which gives the desired formula. It leads to \[\sigma\cdot\nu_s=(-1)^{\kappa(s)}\sigma\cdot\mu_s=(-1)^{\kappa(s)}\epsilon(\sigma)\epsilon(\beta)(-1)^{\sum_{i=1}^np_i(\sigma(i)-i)}\mu_{\sigma\circ s} \ . \] One can prove that  \[(-1)^{\kappa(\sigma\circ s)}=(-1)^{\kappa(s)}(-1)^{\sum_{i=1}^np_i(\sigma(i)-i)} \ , \] which gives \[ \sigma\cdot\nu_s= (-1)^{\kappa(\sigma\circ s)}\epsilon(\sigma)\epsilon(\beta)\mu_{\sigma\circ s}=\epsilon(\sigma)\nu_{\sigma\circ s} \ . \] We get the result for any surjection $f \in \Surj(m,n)$ using its decomposition $f = s \circ \alpha$ and the formula of the right action. 

\noindent Let us prove Point (b). Let $s\in\OSurj(m,n)$ and $t\in\OSurj(m',n')$ be order-preserving surjections. It is sufficient to prove that
\[\kappa(s\times t)= \kappa(s)+\kappa(t)+ (m'-n') m \ .\] We have
$\kappa(s\times t)=\kappa(s)+\kappa(t)+\sum_{i=1}^{n'} (q_i-1)m$, for $q_i=|t^{-1}(i)|$, which gives the desired result.

\noindent Let us prove Point (c). For any $h\in\Surj(l,m)$, the surjection $\alpha \circ h$ decomposes as $t \circ \beta$ with $t$ an order-preserving surjection and $\beta$ an unshuffle. We have 
\[\nu_f \diamond \nu_h =\epsilon(\alpha) (\nu_s\cdot\alpha)\diamond \nu_h=\nu_s\diamond \nu_{\alpha\circ h}=\epsilon(\beta)(\nu_{s}\diamond\nu_t)\cdot\beta . \] Hence, it is sufficient to prove the result for order-preserving surjections $s\in\OSurj(m,n)$ and $t\in\OSurj(l,m)$. Let us denote \[
m_1=|s^{-1}(1)|, \quad m_2=m-m_1, \quad l_1=\sum_{i=1}^{m_1}|t^{-1}(i)|  \quad \mbox{and} \quad l_2=l-l_1 \ .\]
We proceed by induction on $n$. For $n=1$, it is a consequence of the operadic composition in $\SCom$ and the definition of $\kappa$. We decompose $s=s_1\times s_2$ with 
$s_1\in\OSurj(m_1,1)$, $s_2\in\OSurj(m_2,n-1)$ and $t=t_1\times t_2$ with $t_1\in\OSurj(l_1,m_1)$ and $t_2\in\OSurj(l_2,m_2)$. We have  \begin{align*}
    \nu_s\diamond \nu_t & =\nu_{s_1\times s_2}\diamond \nu_{t_1\times t_2} \\ & =(-1)^{\mathrm{d}(\nu_{s_2})m_1+ \mathrm{d}(\nu_{t_2})l_1}(\nu_{s_1}\otimes\nu_{s_2})\diamond(\nu_{t_1}\otimes\nu_{t_2}) \\ & 
=(-1)^{ \left( \mathrm{d}(\nu_{s_2}) + \mathrm{d} (\nu_{t_2}) \right) l_1} \nu_{s_1\circ t_1}\otimes\nu_{s_2\circ t_2} \\ 
 &  =  \nu_{(s_1\circ t_1)\times (s_2\circ t_2)} \\
& = \nu_{s\circ t}\ .\qedhere
\end{align*}

	 \end{proof}

\section{\textcolor{bordeau}{On subcategories of the Karoubi envelope of a prop}}\label{2}

In this section, we recall the construction of the Karoubi envelope of a category (also called idempotent completion or pseudo-abelian hull), see \cite[Section~1.2]{Kar68}. This construction aims to add objects and morphisms to a category so that every idempotent, i.e. every endomorphism $e_A: A \to A$ satisfying $e_A \diamond e_A =e_A$ , is split.

\begin{definition}[Karoubi envelope] \label{Karoubi-env}
The \emph{Karoubi envelope} of a given category $\C$ is the category $\Kar(\C)$ defined by the following data. 
\begin{itemize}
	\item[$\centerdot$] The objects are the pairs $(A, e_A)$ where $A$ is an object in $\C$ and $e_A: A \to A$ is an \emph{idempotent}.
	\item[$\centerdot$] A morphism $f : (A,e_A) \to (B,e_B)$ between two objects in $\Kar(\C)$ is a morphism $f: A \to B$ in $\C$ such that $f= f \diamond e_A = e_B \diamond f= e_B \diamond f\diamond e_A\ .$
	\item[$\centerdot$] The composition is the same as the one in $\C$ and the identity morphism on $(A,e_A)$ is $e_A$.
\end{itemize} 
There is a fully faithful functor  $\eta_\C: \C \to \Kar(\C)$ defined as follows. For an object $A$ in $\C$, we have that $\eta_\C(A)=(A, 1_A)$ and for a morphism $f$ in $\C$, we have that $\eta_\C(f)= f$. 
\end{definition}


\begin{remark}\label{R:inKaroubi}  A morphism $f: A \to B$ in a category $\C$ is a morphism in $\Kar(\C)$ from $(A,e_A)$ to $(B,e_B)$ if and only if $f= e_B \diamond f \diamond e_A$.
\end{remark}

\begin{proposition} \label{prop on Karoubi}
Suppose that $(\C,\otimes,I)$ is a symmetric monoidal structure on $\C$.  Then $(\Kar(\C),\tilde\otimes, (I,1_I))$ is a symmetric monoidal category with
\begin{align*}
    (A,e_A)\ \tilde\otimes \ (B,e_B) &=(A\otimes B,e_A\otimes e_B) \\
    f\ \tilde\otimes \ g&= f\otimes g \ .
    \end{align*}
\end{proposition}

\begin{proof} We observe that $e_A\otimes e_B$ is an idempotent. For $f:(A,e_A)\to (A',e_{A'})$ and $g:(B,e_{B})\to (B',e_{B'})$ two morphisms in $\Kar(\C)$, we have that $f\ \tilde\otimes \ g$ is a morphism from $(A,e_A)\ \tilde\otimes \ (B,e_B)$ to 
$(A',e_{A'})\ \tilde\otimes \ (B',e_{B'})$ since 
\[(e_{A'}\otimes e_{B'})\diamond (f\otimes g)\diamond (e_A\otimes e_B)=(e_{A'}\diamond f\diamond e_A)\otimes (e_{B'}\diamond g\diamond e_B)=f\otimes g \ .\]
If $\rho$ is the left unitor in $\C$, that is, a natural isomorphism defined  by $\rho_A:A\otimes I\to A$, then by naturality one has $e_A\diamond\rho_A=\rho_A\diamond(e_A\otimes 1_I)$. Hence 
$\tilde\rho_{(A,e_A)}:=e_A\diamond\rho_A\diamond (e_A\otimes 1_I)$ is a 
well defined morphism in $\Kar(\C)$ from $(A,e_A)\tilde\otimes (I,1_I)$ to $(A,e_A)$. It is natural and it is an isomorphism. The associator, the right unitor and the swap map are defined similarly, and  the desired diagrams between the natural transformations (the commutative pentagon, hexagon, and triangle) commute.
\end{proof}



Let $\C$ and $\D$ be two categories. A \emph{semifunctor} $F:\C \to \D$ maps objects (resp. arrows) in $\C$ to objects (resp. arrows) in $\D$, preserving domain, codomain and composition, see \cite[Section~4]{Mitchell72}. The difference between semifunctors and functors is that semifunctors need not preserve identities.

\begin{definition}
There is a forgetful semifunctor $\epsilon_\C: \Kar(\C) \to \C$ defined as follows. For $(A,e_A)$ an object of $\Kar(\C)$, we have that $\epsilon_\C(A,e_A)=A$ and any morphism is sent to itself. 
In particular, the identity $e_A$ of $(A,e_A)$ is sent to $e_A$ which is not equal to $1_A$ in general.
\end{definition}

Let $(\Co, \otimes, 1)$ be a prop. By Remark \ref{phi}, the morphism $\varphi: {\K}[\mathbb{S}_n]\to\Co(n,n)$, maps any idempotent $e$ in ${\K}[\mathbb{S}_n]$ to an idempotent in $\Co(n,n)$. Via an abuse of notation, we will denote the idempotents $\varphi(e)$ by $e$, so that by relation 
(\ref{E:action_comp}), the action $\cdot$ of the symmetric group on $\C$ corresponds to the composition $\diamond$ in the prop $\C$.

\noindent There is a well-known construction of a set of primitive orthogonal idempotents of  ${\K}[\mathbb{S}_n]$ indexed by the partitions of $n$ (see \cite[Section 4.1]{FultonHarris}). For $\lambda$ a partition of $n$, we denote by $e_\lambda$ the associated idempotent. For example, we have:\[e_{(1^n)} \coloneqq \frac{1}{n!}\sum_{\sigma\in\mathbb{S}_n}\epsilon(\sigma)\sigma \quad  \text{ and } \quad   e_{(n)}  \coloneqq \frac{1}{n!}\sum_{\sigma\in\mathbb{S}_n}\sigma \ \] In this section, we are interested in particular subcategories of $\Kar(\Co)$ built from the idempotents $e_{(1^n)}$.



\begin{definition} \label{Lambda-C}
Let $(\Co, \otimes, 1)$ be a prop. We denote by $\Lambda \Co$ the full subcategory of $\Kar(\Co)$ whose objects are given by $(n, e_{(1^n)})$ for all $n \in \mathbb{N}$. 
\end{definition}

\noindent By Definition \ref{Karoubi-env} and Remark \ref{R:inKaroubi}, an element of $\Lambda\Co((m, e_{(1^m)}),(n, e_{(1^n)}))$ is necessarily of the form 
	\[e_{(1^n)}\diamond f\diamond e_{(1^m)}=\frac{1}{n!m!}\sum_{\sigma\in\mathbb{S}_n,\tau\in\mathbb{S}_m} \epsilon(\sigma)\epsilon(\tau)  \sigma\cdot f\cdot\tau \ ,\] with $f\in\Co(m,n)$. In what follows, we denote objects in $\Lambda\C$ simply $n$ instead of $(n,e_{(1^n)})$. \medskip

\noindent In order to  give  in Section  \ref{prop-partition},  an explicit description of the composition of partitions we introduce the category $\Co_{\Lambda}$ which is equivalent to $\Lambda \Co$.   More generally, if
the graded vector spaces $\Co(m, n)$ are endowed with a basis such that composition of elements
of the basis is up to a sign an element of the basis, then the combinatorics of the category $\Co_{\Lambda}$
is more easily understood than that of $\Lambda\C $.

\begin{definition}
Let $(\Co, \otimes, 1)$ be a prop. We denote by $\Co_{\Lambda}(m,n)$  the quotient of $\Co(m,n)$ by the relation \[f \sim \epsilon(\sigma)\epsilon(\tau) \tau\cdot f\cdot\sigma \ , \quad \mbox{for} \; \tau\in\mathbb{S}_n,\sigma\in\mathbb{S}_m \ .\] 
\end{definition}

\noindent The normalization map $f \mapsto e_{(1^n)} \diamond f \diamond e_{(1^m)}$ induces an isomorphism \[\C_{\Lambda}(m,n) \stackrel{\cong}{\to} \Lambda \C(m,n) \ ,\] for all $m,n\in\N$. Via these isomorphisms, one can define a category $\C_{\Lambda}$ isomorphic to the category $\Lambda\C$.
Next Theorem gives the explicit composition in $\C_\Lambda$ induced by that in  $\Lambda\C$.

\begin{theorem}\label{T:catquotient} Let $\C$ be a prop and $\varphi:\K[\Ss_n]\to\C(n,n)$ be the map induced by the action of the symmetric group. For $[f]\in \C_\Lambda(m,n)$ and  $[g]\in\C_\Lambda(l,m)$, the composition $[f]\tb [g]$ in the category $\C_\Lambda$ has the following form
	\[[f]\tb [g]=\frac{1}{m!}\sum_{\sigma\in\mathbb{S}_m}\epsilon(\sigma)[f\diamond\varphi(\sigma)\diamond g] \  . \] 
\end{theorem}

\begin{proof}
	For all $n \in \mathbb{N}$, let us denote $e_n \coloneqq e_{(1^n)}$. Let $f\in\C(m,n)$ and $g\in\C(l,m)$. The normalization map sends any $\left[f\right]$ in $\C_\Lambda(m,n)$ to $e_n\diamond f\diamond e_m$ in $\Lambda C(m,n)$.  The composition of $e_n\diamond f\diamond e_m$ with $e_m\diamond g\diamond e_l$ in the category $\Lambda \C$, or equivalently in the category $\C(m,n)$ is :
	\[e_n\diamond f\diamond e_m\diamond e_m\diamond g\diamond e_l=e_n\diamond f\diamond e_m\diamond g\diamond e_l\]
	which is the image of $[f\diamond e_m\diamond g]\in \C_\Lambda(l,n)$ by the normalization map.
	This leads to 
	\[[f]\tb [g]=[f\diamond e_m\diamond g].\] In addition, we have $[1_m]\tb [g]=[e_m\diamond g]=[g]$ and $[f]\tb [1_m]=[f\diamond e_m]=[f]$.
\end{proof}

\begin{remark} Note that the family of quotient maps $\C(m,n)\to\C_\Lambda(m,n)$ does not provide a functor from $\C$ to $\C_\Lambda$.  The equivalence of categories $\Lambda\C\to\C_\Lambda$ is, on morphisms, the composition of the restriction of the semi-functor \[ \epsilon_\C:\Kar(\C)((m,e_{(1^m)}),(n,e_{(1^n)}))\to\C(m,n) \] to the category $\Lambda\C((m,e_{(1^m)}),(n,e_{(1^n)}))$ with the quotient map.
\end{remark}

\section{\textcolor{bordeau}{A graded linear prop spanned by partitions}} \label{prop-partition}

In Section \ref{2}, we associate to every prop $\C$ a category $\C_\Lambda$, which is equivalent to a subcategory $\Lambda\C$ of the Karoubi 
envelope of $\C$. This section focuses on the particular case of the prop $\E=\Omega\SCom$ (see Section \ref{1}) and its associated category $\E_\Lambda$ (see Section \ref{2}). The Karoubi envelope $\Kar(\E)$ inherits a prop structure from that of $\E$, see Proposition \ref{prop on Karoubi}. However, it does not induce a prop structure at the level of a given subcategory of $\Kar(\E)$ in general. For $\Lambda \E$, a necessary condition would be that \[e_{(1^m)}\otimes e_{(1^n)}=e_{(1^{m+n})} \ , \] in $\E$ which is not the case. Nonetheless, in Section \ref{prop-ELambda}, we extend the category $\E_\Lambda$ into a graded linear prop by introducing another monoidal product $\odot$.

\subsection{The composition in the category $\E_{\Lambda}$}

In this section, we describe the category structure $\E_{\Lambda}$ given by Theorem \ref{T:catquotient} in the case of $\C = \E$. We start by proving that this category is spanned by partitions.

\begin{proposition} \label{E_Lambda=part}
	For all $n,m \in \N$, there is an isomorphism \[\E_{\Lambda}^\bullet(m,n) \cong \left\{
	\begin{array}{ll}
		\K \left[\Part \left(m,n \right) \right]    & \mbox{if } \bullet = m-n, \text{ and } m\geqslant  n, \\
		0 &  \mbox{otherwise } 
	\end{array}
	\right. \] where $\Part \left(m,n \right) $ denotes the partitions of $m$ into $n$ parts. 
\end{proposition}

\begin{proof}
    The result follows directly from Lemma \ref{equivalences}. 
\end{proof}

\begin{lemma}\label{equivalences}
For all $f,g\in\Surj(m,n)$, the following propositions are equivalent : 
\begin{enumerate}
    \item[($i$)] $[\nu_f]=[\nu_g] \in \E_\Lambda^{m-n}(m,n) \ ;$
    \item[($ii$)] there exist $\tau \in \Ss_m$ and $\sigma\in\Ss_n$ such that $g=\sigma\circ f\circ\tau$; 
    \item[($iii$)] $\proj(f)=\proj(g)$,
\end{enumerate} where $\proj:\Surj(m,n)\to\Part(m,n)$ is the map defined in Notation (\ref{N:part}). 
\end{lemma}

\begin{proof} 
We have $[\nu_f]=[\nu_g]$ if and only if there exist $\tau\in\Ss_m$ and $\sigma\in\Ss_n$ such that \[\nu_g=\epsilon(\sigma)\epsilon(\tau)\sigma\cdot \nu_f\cdot\tau=\nu_{\sigma\circ f\circ\tau} \ . \] This proves the equivalence ($i$)$\iff$($ii$). The implication ($ii$) $\Longrightarrow$ ($iii$) is immediate. Conversely, we assume that $\proj(f)=\proj(g)$ and let us denote by $\lambda_1\geqslant \ldots\geqslant \lambda_n$ this partition. 
There is a unique $s\in\OSurj(m,n)$ such that $|s^{-1}(i)|=\lambda_i$, for all $i$ and there exists a permutation $\beta\in\Ss_n$ such that  
$|f^{-1}(i)|=\lambda_{\beta(i)}$. In particular, the surjection $\beta\circ f$ writes uniquely as $\beta\circ f=s\circ u$ with $u\in\Sh_{\beta\circ f}$, 
so that $f$ decomposes as $\beta^{-1}\circ s\circ u$. Similarly, there exist $\gamma \in\Ss_n,v\in\Ss_m$ such that $g=\gamma^{-1}\circ s\circ v$.
By construction, $\sigma = \gamma^{-1} \circ \beta$ and $\tau = u^{-1} \circ v$ are such that $g=\sigma\circ f\circ \tau$, which proves ($iii$) $\Longrightarrow$ ($ii$). 
\end{proof}

\begin{notation} \rm 
For a given partition $\lambda\in\Part(m,n)$, we denote \[\rho_\lambda \coloneqq [\nu_f] \in\E_\Lambda^{m-n}(m,n) \] the class of $\nu_f$, where $f$ is any surjection such that $\proj(f)=\lambda$.
\end{notation}

\begin{theorem}\label{T:compLambdaLC} The composition of two basis elements in the category $\E_{\Lambda}$ is a weighted average of basis elements, i.e. for  $\lambda\in\Part(m,n)$ and $\mu\in\Part(l,m)$, we have
\[\rho_\lambda\tb \rho_\mu=\sum_{\alpha\in\Part(l,n)} c_{\alpha}^{\lambda,\mu} \rho_\alpha \]
where $ m!c_{\alpha}^{\lambda,\mu}\in\N$ and  $\sum_\alpha c_\alpha^{\lambda,\mu}=1.$ The identity morphisms are given by $\rho_{(1^m)} \in \E_{\Lambda}(m,m)$.
\end{theorem}

\begin{proof} Let $f\in\Surj(m,n)$ and $g\in\Surj(l,m)$  such that $\proj(f)=\lambda$ and $\proj(g)=\mu$. By Theorem \ref{T:catquotient}
	we have
	\[[\nu_f]\tb [\nu_g]=\frac{1}{m!}\sum_{\sigma\in\mathbb{S}_m}\epsilon(\sigma) [\nu_{f}\diamond\varphi(\sigma)\diamond \nu_g]\]    
	and $\varphi(\sigma)=\epsilon(\sigma) \nu_\sigma$, see Notation \ref{nu}. This leads to 
	\begin{equation}\label{E:comp}
	     [\nu_f]\tb [\nu_g]=\frac{1}{m!}\sum_{\sigma\in\mathbb{S}_m} [\nu_{f\circ\sigma\circ g}] \ . \qedhere
      \end{equation}
\end{proof}

\begin{corollary} \label{P:keycompG} 
For any partition $\lambda=\lambda_1\geqslant \ldots\geqslant \lambda_n\in\Part(m,n)$, we have
\[\rho_\lambda\tb \rho_{(2,1^{m-1})}=\frac{1}{m}\sum_{i=1}^{n}\lambda_i [\nu_{s(\lambda_1,\ldots,\lambda_{i-1},\lambda_i+1,\lambda_{i+1},\ldots,\lambda_n)}] \ ,\] where $s(p_1,\ldots,p_n)$ denotes the unique order-preserving surjection $s \in \OSurj(m,n)$ such that $|s^{-1}(i)|=p_i$ and $m = p_1 + \cdots + p_n$. \label{s()}
\end{corollary}

\begin{proof} 
Let $f=s(\lambda_1,\ldots,\lambda_n)\in\OSurj(m,n)$ and $g=s(2,1,\ldots,1)\in\OSurj(m+1,m).$ By Formula (\ref{E:comp}),
we have
\[\rho_\lambda\tb \rho_{(2,1^{m-1})}=\frac{1}{m!}\sum_{\sigma
\in\Ss_m} [\nu_{f\circ\sigma\circ g}]\ .\]
For $\sigma\in\Ss_m$, there is a unique decomposition 
$\sigma\circ g=t \circ u$ with $t\in\OSurj(m+1,m)$ and $ u \in\Sh_{\sigma\circ g}$. For all $1\leqslant  j\leqslant  m$, we have that \[|t^{-1}(j)|=2 \; \iff \; \sigma(1)=j \ . \] There are $(m-1)!$ permutations $\sigma$ such that 
$\sigma(1)=j$. In that case, we have $t=s(1^{j-1},2,1^{m-j})$. 
By Lemma \ref{equivalences}, we have $[\nu_{f\circ t\circ u}]=[\nu_{f\circ t}]$ , so that
\[\rho_\lambda\tb \rho_{(2,1^{m-1})}=\frac{1}{m}\sum_{j=1}^{m}
[\nu_{f\circ s(1^{j-1},2,1^{m-j})}].\]
Finally, we decompose the sum as
\[\sum_{j=1}^{m}
[\nu_{f\circ s(1^{j-1},2,1^{m-j})}]=
\sum_{i=1}^n\sum_{j= \lambda_1+\cdots+\lambda_{i-1}+1}^{\lambda_1+\cdots+\lambda_{i-1}+\lambda_i} [\nu_{f\circ s(1^{j-1},2,1^{m-j})}] \ .\] We conclude by noting that for $\lambda_1+\ldots+\lambda_{i-1}+1 \leqslant  j\leqslant  \lambda_1+\ldots+\lambda_{i-1}+\lambda_i$, we have
\[[\nu_{f\circ s(1^{j-1},2,1^{m-j})}]=[\nu_{s(\lambda_1,\ldots,\lambda_{i-1},\lambda_i+1,\lambda_{i+1},\ldots,\lambda_n)}] \ . \qedhere \] 
\end{proof}

\begin{example} The formula of Corollary \ref{P:keycompG} gives
\[\rho_{(3,3,1)}\tb \rho_{(2,1^{6})}=
\frac{1}{7}(3[\nu_{s(4,3,1)}]+3[\nu_{s(3,4,1)}]+[\nu_{s(3,3,2)}])=\frac{1}{7} (6\rho_{(4,3,1)}+\rho_{(3,3,2)}) \ .\] 
\end{example}

\subsection{A particular prop structure}\label{prop-ELambda}

In this section, we introduce a particular monoidal product $\odot$ which turns the category $\E_\Lambda$ into a prop. 

\begin{notation}	
	Let us define a family of elements $P_{m,n} \in \E^{m-n}_\Lambda(m,n)$ for $m\geqslant  n$ as follows:
	\begin{enumerate}
	\item $P_{m,m}= \rho_{(1^m)} \in$ $\E^0_\Lambda(m,m)$, \smallskip
		\item $P_{m,m-1}= \rho_{(2,1^{m-2})} \in$  $\E^1_\Lambda(m,m-1)$, 
  \smallskip
		\item $ P_{m,n}=  P_{n+1,n}\tb\ldots\tb P_{m,m-1}\in$ $\E^{m-n}_\Lambda(m,n)$.	
	\end{enumerate}
\end{notation}


\begin{lemma}\label{L:P_descr} The element $P_{m,n}\in\E^{m-n}_\Lambda(m,n)$ satisfies
\[P_{m,n}=\frac{1}{|\OSurj(m,n)|}\sum_{s\in\OSurj(m,n)}[\nu_s]=\frac{1}{|\OSurj(m,n)|}\sum_{s\in\OSurj(m,n)}\rho_{\proj(s)} \ .\] 
\end{lemma}

\begin{proof}
Let us prove the result by induction on $m\geqslant  n$. For $m=n$, we have by definition $P_{n,n}=\rho_{(1^n)}$ and there is only one $s\in\OSurj(n,n)$ such that $\proj(s)=(1^n)$. Similarly, $\Part(n+1,n)$ has only one element, so that the formula boils down to \[ P_{n+1,n}=\rho_{(2,1^{n-1})} \ .\]
Assume the formula is true for $P_{m,n}$. We have $P_{m+1,n}=P_{m,n} \tb P_{m+1,m}$, that is,
\[P_{m+1,n}=\frac{1}{|\OSurj(m,n)|}\sum_{s\in\OSurj(m,n)}[\nu_s]\tb \rho_{(2,1^{m-1})}\ .\]
By Theorem \ref{T:compLambdaLC}, we know that $P_{m+1,n}$ is a weighted average of the elements $\rho_\lambda$, 
where $\lambda$ runs in $\Part(m+1,n)$. For $s=s(p_1,\ldots,p_n)$ the order-preserving surjection such that $|s^{-1}(i)|=p_i$, 
we have
\[[\nu_s]\tb \rho_{(2,1^{m-1})}=\frac{1}{m}\sum_{i=1}^{n}p_i [\nu_{s(p_1,\ldots,p_{i-1},p_i+1,p_{i+1},\ldots,p_n)}]\ ,\] by Proposition \ref{P:keycompG}. 
Let $s(q_1,\ldots,q_n)\in \OSurj(m+1,n)$. In the composition $P_{m,n}\tb \rho_{(2,1^{m-1})} \ ,$ the element  $[\nu_{s(q_1,\ldots,q_n})]$ appears with weight
\[\frac{1}{m|\OSurj(m,n)|}\sum_{i=1}^n(q_i-1)=\frac{m+1-n}{m|\OSurj(m,n)|}\,\]
which is independant of the chosen order-preserving surjection in $\OSurj(m+1,n)$. This gives the desired result. We also obtain that for $m +1 > n$,
\[|\OSurj(m+1,n)|=|\OSurj(m,n)|\frac{m}{m+1-n} \ . \qedhere \]
\end{proof}

\begin{remark} One can also express $P_{m,n}$ in the basis $\rho_\lambda$ with $\lambda\in \Part(m,n)$, by counting the number of order-preserving surjections $s$ such that $\proj(s)=\lambda$. As an example, we have
    \[P_{6,3}=\frac{1}{10}(3\rho_{(4,1,1)}+6\rho_{(3,2,1)}+\rho_{(2,2,2)}) \ .\]
\end{remark}

\begin{theorem}\label{T:main} For $\alpha\in\Part(m,n)$ and $\beta\in\Part(m',n')$, the monoidal product defined as
	\[ \rho_\alpha\odot \rho_\beta=(-1)^{\mathrm{d}(\rho_{\alpha})n'} P_{m+m',n+n'} \ , \] endows the category $\E_\Lambda$ with a structure of symmetric monoidal category.
\end{theorem}

\begin{proof} 
It is immediate to prove that the product $\odot$ is associative. Let us prove the compatibility between the composition $\ast$ and the monoidal product $\odot$. Let us consider 
\[\lambda \in\Part(m,n), \ \lambda_2\in\Part(m_2,n_2), \ \beta\in\Part(l,m), \ \mbox{and} \ \beta_2 \in\Part(l_2,m_2) \ . \] 
On the one hand, we have \[  (\rho_{\lambda }   \odot\rho_{\lambda_2})\tb(\rho_\beta\odot\rho_{\beta_2}) = (-1)^{\mathrm{d}(\rho_{\lambda})n_2+\mathrm{d}(\rho_{\beta})m_2}P_{l+l_2,n+n_2} \ , \]
	and on the other hand, we have 
\[ 	(\rho_{\lambda }\tb\rho_\beta)\odot (\rho_{\lambda_2}\tb\rho_{\beta_2})  =\sum_{\mathclap{\substack{ \alpha\in\Part(l,n)\\ \alpha_2\in\Part(l_2,n_2)}}}c_\alpha^{\lambda ,\beta}c_{\alpha_2}^{\lambda_2,\beta_2} \rho_\alpha\odot \rho_{\alpha_2}  = 
		(-1)^{(l-n)n_2}(\,\sum_{\mathclap{\substack{ \alpha\in\Part(l,n)\\ \alpha_2\in\Part(l_2,n_2)}}}c_\alpha^{\lambda ,\beta}c_{\alpha_2}^{\lambda_2,\beta_2}\,) P_{l+l_2,n+n_2} \ .  \] By Theorem \ref{T:compLambdaLC}, we obtain that \[ (\rho_{\lambda }   \odot\rho_{\lambda_2})\tb(\rho_\beta\odot\rho_{\beta_2}) =  (-1)^{\mathrm{d}(\rho_{\beta})\mathrm{d}(\rho_{\lambda_2})}(\rho_{\lambda }\tb\rho_\beta)\odot (\rho_{\lambda_2}\tb\rho_{\beta_2}) \ .   \]


\noindent Finally, let us set \[s_{m,m'} \coloneqq (-1)^{mm'}\rho_{(1^{m+m'})} \in \E_\Lambda(m+m',m+m') \ . \]  We have
 \[(-1)^{\mathrm{d}(\rho_{\alpha})\mathrm{d}(\rho_{\beta})} \left(\rho_\beta\odot \rho_\alpha \right) \tb s_{m,m'} = s_{n,n'} \tb \left(\rho_\alpha\odot \rho_\beta\right)=(-1)^{mn'}P_{m+m',n+n'} \ . \qedhere \] 
\end{proof}

\begin{remark}
The underlying operad associated to the prop $\E_\Lambda$ is isomorphic to the operad $\SCom$ via the isomorphism given by \[\rho_{(m)}\in\E_\Lambda(m,1) \longmapsto (-1)^{\frac{m(m-1)}{2}}\mu_m \ . \] For dimension reasons, the prop $\E_\Lambda$ is nevertheless
 not isomorphic to the prop freely generated by the operad $s \Com$, and thus to $\mathcal{E}$. Let us note that the prop $\E_\Lambda$ is also not finitely generated since any family of generators would necessary contain a generator in each $\E^1_\Lambda(m,m-1)$. 
\end{remark}

\noindent Additional computations suggest the following conjecture.

\begin{conjecture}\label{conj} The prop structure given by Theorem \ref{T:main} is the unique prop structure on the category $\E_\Lambda^\bullet$ which is the sum of integers on objects and such that for all $m\geqslant  0$, 
\[ \rho_{(1)}\odot \rho_{(1^m)}=\rho_{(1^{m+1})} \quad \mbox{and} \quad 
 \rho_{(1)}\odot\rho_{(2,1^m)}=\rho_{(2,1^{m+1})} \ .\]
\end{conjecture}

\section{\textcolor{bordeau}{Relation with functor homology on free groups}} \label{Ext}

This work is motivated by the relation between the prop freely generated by the operadic suspension of $\Com$ and the extension groups between the tensor powers of the abelianisation functor obtained in \cite{vespa15}. More precisely, let $\textbf{gr}$ be the category of finitely generated free groups, $\textbf{ab}$ the category of finitely generated free abelian groups and $\mathrm{Vect}_\K$ the category of $\K$-vector spaces. We denote by $\F \left(\textbf{gr} \right)$ the category of functors from $\textbf{gr}$ to $\mathrm{Vect}_\K$. Let $\mathfrak{a} : \textbf{gr} \longrightarrow \textbf{ab}$ be the abelianization functor and let $T^n : \mathrm{Vect}_\K \longrightarrow \mathrm{Vect}_\K$ be the $n^{th}$ tensor product functor. Let us consider the functor \[\mathfrak{a}_{\K} \coloneqq \mathfrak{a} \underset{\mathbb{Z}}{\otimes} \K :\textbf{gr} \longrightarrow \mathrm{Vect}_\K \ . \] 
 We consider the category $\E_{T}$ enriched in graded $\K$-vector spaces whose objects are the integers $\mathbb{N}$ and whose morphisms are given by \[\E^{\bullet}_T(m,n) \coloneqq \Ext^{\bullet}_{\F \left(\textbf{gr}\right)} \left(T^n \circ \abK ,   T^m \circ \abK\right) \] with the composition given by the Yoneda product. It was shown in \cite{vespa15} that, together with the external product of extensions, this category forms a prop which is freely generated by its underlying operad. This operad is identified in \cite[Section~9]{KV23} as being the suspension of $\Com$. The authors introduce the generator \[[\pi^{\otimes n}] \in \E_{T}^{n-1}(n,1)  \] and they define an isomorphism of props $ \iota :  \E_{T} \to \E$ given by \[[\pi^{\otimes p_1} ]\otimes \ldots \otimes [\pi^{\otimes p_n} ] \in \E_{T}(m,n) \longmapsto \mu_{p_1} \otimes \ldots \otimes \mu_{p_n} \ . \] 
By Section \ref{E} we have an isomorphism $\E \cong \E$ given for every surjection $f \in \Surj(m,n)$  by,
\[\begin{array}{cccc}
    \E_f^{m-n} & \longrightarrow  & \E_f^{m-n} \\
    \mu_{p_1}\otimes\ldots\otimes \mu_{p_n}& \longmapsto  & \nu_f \  \qquad & 
    \ .
\end{array} \] The composition of $\iota$ with the previous isomorphism leads to a natural equivalence: 
\begin{equation*} 
 \Kar(\E_{T}) \overset{\sim}{\longrightarrow} \Kar(\E) \ .
\end{equation*}  Let $\Lambda^n : \mathrm{Vect}_\K\longrightarrow \mathrm{Vect}_\K$ be the $n^{th}$ exterior power functor. We consider the category $\Lambda \E_T$ whose objects are the integers and whose morphisms are given by
 \[\Lambda \E_T(m,n) = \Ext^{\bullet}_{\F \left(\textbf{gr} \right)}\left( \Lambda^n \circ \abK, 
 \Lambda^m \circ \abK \right) \ , \] with the composition given by the Yoneda product. The prop isomorphism $\E_{T} \simeq  \E$ induces isomorphisms \[\Lambda \E_T \cong \Lambda \E \cong \E_\Lambda \ .\] Thus, we recover the following isomorphism established in \cite[Theorem~4.2]{vespa15}
 \[\Ext^\bullet_{\Fgr}(\Lambda^n\circ\abK,\Lambda^m\circ \abK) \cong \left\{
\begin{array}{ll}
\K [ \Part(m,n) ] & \mbox{if } \bullet = m-n \\
0 &  \mbox{otherwise } , 
\end{array}
\right. \] and $\Lambda \E_T$ inherits a prop structure from that of $\E_{\Lambda}$. The Yoneda product of $\Lambda \E_T$ coincides with the composition and Theorem \ref{T:compLambdaLC} thus gives an explicit formula. 

\begin{remark}
   We could hope that the monoidal product of $\Lambda \E_T$ corresponds to an external product defined on extensions of exterior powers functors, as it was the case for $\E$. However, this is not the case: exploiting the Hopf structure of the functor $\Lambda$, one can actually define an external product on $\Lambda \E_T$ but it is not compatible with the Yoneda product.
\end{remark}


\noindent Recall that the regular representation of the symmetric group $\Ss_d$ decomposes as 
\begin{equation*} 
\K [\Ss_d] \cong \bigoplus_{\lambda \vdash d} S_\lambda^{\oplus \mathrm{dim} (S_\lambda)},
\end{equation*}
where $S_\lambda$ is the simple module indexed by the partition $\lambda$ of $d$. Moreover, a simple $\K[\mathbb{S}_d]$-module is isomorphic to $S_\lambda$ for a unique partition $\lambda$ of $ d$. We deduce from the previous decomposition that
\begin{equation*} \label{E:decompT}
T^d \circ \abK \cong \bigoplus_{\lambda \vdash d} \mathbf{S}_\lambda^{\oplus \dim (S_\lambda)}
\end{equation*}
where \[\mathbf{S}_\lambda=(T^d \circ \abK)\underset{\mathbb{S}_d}{\otimes} S_\lambda \ \] is a simple functor. For $\lambda$ a partition of $m$ and $\mu$ a partition of $n$ we have:
\[\Kar(\E_{T})((m,e_\lambda),(n,e_\mu))=e_\mu \cdot \Ext^{\bullet}_{\F \left(\textbf{gr}\right)} 
\left(T^n\circ\abK,T^m\circ\abK\right) \cdot e_\lambda = \Ext^{\bullet}_{\F \left(\textbf{gr} \right)} \left(\mathbf{S}_\mu ,   \mathbf{S}_\lambda\right) \ , \] and thus
\begin{equation}\label{E:main_idempotent}
\Ext^{\bullet}_{\F \left(\textbf{gr} \right)} \left(\mathbf{S}_\mu ,   \mathbf{S}_\lambda\right) \simeq e_\mu \cdot \E(m,n) \cdot e_\lambda \ .
\end{equation} Using Formula (\ref{E:main_idempotent}) and the generators $\nu_f$ introduced in Notation \ref{nu}, one can compute these extension groups. We conclude by giving some examples below.

\begin{example}\label{exterior powers}
For all $n$ and all $m$, we have \[\Lambda \E_T(m,n)  = 
 \Ext^{\bullet}_{\F \left(\textbf{gr} \right)} \left(\mathbf{S}_{(1^n)} ,   \mathbf{S}_{(1^m)}\right) \ . \]  For example,  we have
 \[\Ext^1_{\F \left(\textbf{gr}\right)} \left(\mathbf{S}_{(1^2)} ,   \mathbf{S}_{(1^3)} 
 \right) \simeq e_{(1^2)} \cdot \E(3,2) \cdot e_{(1^3)},\ \text{ with }\] 
 \[e_{(1^2)}=\dfrac{1}{2}(1-\tau_{1,2})\ \text{ and }\ e_{(1^3)}=\dfrac{1}{6} \sum_{\sigma \in \mathbb{S}_3}\epsilon(\sigma)\sigma\ .\]
By Theorem \ref{P:easyprop}, any surjection $f\in \Surj(3,2)$ satisfies
    \[e_{(1^2)} \cdot \nu_f \cdot e_{(1^3)}=\dfrac{1}{6} \underset{g \in \mathrm{Surj}(3,2)}{\sum} \nu_g\ , \]
    confirming that $\Ext^1_{\F \left(\textbf{gr} \right)} \left(\mathbf{S}_{(1^2)} ,   \mathbf{S}_{(1^3)} \right) $ is $1$-dimensional.
\end{example}

\begin{example} We recover the following result of \cite[Theorem 4.2]{vespa15} :
\[\Ext^{\bullet}_{\F \left(\textbf{gr} \right)} \left(T^n \circ \abK,   S^m \circ \abK\right)=\begin{cases} 
\K & \text{ if } n=m \text{ and } \bullet=0,\\
0 & \text {else}.\end{cases}\]
We already know that the graded vector space is concentrated in degree $m-n$. For $m=n$,
\[\Ext^{0}_{\F \left(\textbf{gr}\right)} \left(T^m \circ \abK,   S^m \circ \abK\right) \]
is one dimensional, generated by the image of the idempotent $e_{(m)}$ of $\Ss_m$ via the isomorphism $\K[\Ss_m]\simeq \E(m,m).$
If $m\not=n$, we can assume $m>n$. Working in the Karoubi envelope of $\E_T$, we have
    \[\Kar(\E_{T}) ((m,e_{(m)}), (n,id))=\Ext^{\bullet}_{\F \left(\textbf{gr} \right)} \left(T^n \circ \abK,   S^m \circ \abK\right)\simeq\E(m,n)\cdot e_{(m)}.\]
  Let $f\in\Surj(m,n)$ be a surjection which decomposes as $s\circ\alpha$ with $s\in\OSurj(m,n)$. Since $\nu_f=\epsilon(\alpha)\nu_s\cdot\alpha$, it is enough to prove that $\nu_s\cdot e_{(m)}=0.$ Let us denote $p_i=|s^{-1}(i)|$. Any
   $\sigma\in\Ss_m$ writes uniquely as $\sigma=(\sigma_1\times\ldots\times\sigma_n)\circ u$ with $\sigma_i\in\Ss_i$ and $u\in\Sh_s$, so that $s\circ\sigma=s\circ u$.  Hence \[\nu_s\cdot\sigma=\epsilon(\sigma_1)\ldots\epsilon(\sigma_n)\epsilon(u)\nu_{s\circ u} \ .\]
 If $m>n$, then there exists $i$ such that $p_i>1$. In particular, in $\Ss_i$,  there are  as many odd permutations as even permutations, so that $\nu_s\cdot e_{(m)}=0$.
\end{example}

\begin{example} With the help of Formula (\ref{E:main_idempotent}), let us compute
     \[\ \Ext^1_{\F \left(\textbf{gr} \right)} \left(\mathbf{S}_{(1^2)} , \mathbf{S}_{(2,1)}  \right) \quad \mbox{and} \quad  \ \Ext^1_{\F \left(\textbf{gr} \right)} \left(\mathbf{S}_{(2)} , \mathbf{S}_{(2,1)}  \right) \ . \]
     Following \cite[Section4]{FultonHarris}, we have
     
    \[e_{(1^2)}=\dfrac{1}{2}(1-\tau_{1,2}),\ e_{(2)}=\dfrac{1}{2}(1+\tau_{1,2}) \text{ and } e_{(2,1)}=\dfrac{1}{3} (1-\tau_{1,3}+\tau_{1,2}-(132))\] 
    where $(132)$ is the cyclic permutation. Using Theorem \ref{P:easyprop}, we obtain that 
    $\Ext^1_{\F \left(\textbf{gr} \right)} \left(\mathbf{S}_{(1^2)} , \mathbf{S}_{(2,1)}  \right)$ is $1$-dimensional, generated by
    $e_{(1^2)} \cdot \nu_f \cdot e_{(2,1)}$ with $f\in\Surj(3,2)$ defined by $f(1)=1$ and $f(2)=f(3)=2$. Similarly, we obtain that     $\Ext^1_{\F \left(\textbf{gr} \right)} \left(\mathbf{S}_{(2)} , \mathbf{S}_{(2,1)}  \right)$ is $1$-dimensional, generated by
    $e_{(2)} \cdot \nu_f \cdot e_{(2,1)}$ with $f\in\Surj(3,2)$ defined by $f(1)=1$ and $f(2)=f(3)=2$.
  Indeed, one can check that \[\nu_f\cdot e_{(2,1)}=\frac{1}{3}( \nu_{f}+\nu_{f\circ \tau_{1,3}}-2\nu_{f\circ\tau_{1,2}}) \ . \]

\end{example}

\bibliographystyle{alpha}
\bibliography{EHLVZ}

\end{document}